\pdfoutput=1
\documentclass{article}
\usepackage[a4paper,hscale=.67,vscale=.75]{geometry}
\usepackage[english]{babel}
\usepackage{amsmath,amssymb}
\usepackage{amsthm,thmtools}
\usepackage[boxed,linesnumbered,lined,longend]{algorithm2e}
\usepackage{enumitem}
\usepackage{tikz}
\usetikzlibrary{shapes.geometric}
\usepackage{lmodern}
\usepackage{microtype}
\usepackage[unicode]{hyperref}

\title{Discrete and metric divisorial gonality can be different}
\author{Josse van Dobben de Bruyn \and Harry Smit \and Marieke van der Wegen}
\date{6 July 2021}

\makeatletter
\newcommand{\my@author}[1]{\par\bigskip}
\newcommand{\my@address}[1]{\par\noindent\textsc{#1}}
\newcommand{\my@email}[1]{\par\noindent \textit{E-mail address:} \texttt{\href{mailto:#1}{#1}}}
\AtEndDocument{
	\par\bigskip
	\my@author{Josse van Dobben de Bruyn}
	\my@address{Delft Institute of Applied Mathematics, Delft University of Technology, Mekelweg~4, 2628~CD~Delft, The Netherlands.}
	\my@email{J.vanDobbendeBruyn@tudelft.nl}
	
	\my@author{Harry Smit}
	\my@address{Max Planck Institute for Mathematics, Vivatsgasse 7, 53111 Bonn, Germany.}
	\my@email{smit@mpim-bonn.mpg.de}
	
	\my@author{Marieke van der Wegen}
	\my@address{Department of Information and Computing Sciences, Universiteit Utrecht, Princetonplein~5, 3584~CC~Utrecht, The Netherlands}
	\my@address{Mathematical Institute, Utrecht University, P.O.~Box 80010, 3508~TA~Utrecht, The Netherlands}
	\my@email{M.vanderWegen@uu.nl}
}
\makeatother

\makeatletter
\def\myautoref{\@ifstar\@myautoref\@@myautoref}
\def\@@myautoref#1#2{\hyperref[#2]{\autoref*{#1}\ref*{#2}}}  
\def\@myautoref#1#2{\hyperref[#1]{\autoref*{#1}(#2)}}        
\makeatother

\newcommand{\mysecref}[1]{\hyperref[#1]{\S\ref*{#1}}} 

\def\mylinkcolour{blue!40!black}
\hypersetup{colorlinks=true,linkcolor=\mylinkcolour,urlcolor=\mylinkcolour,citecolor=\mylinkcolour}

\def\schaal{1.7}
\tikzset{vertex/.style={circle,fill,inner sep=1.5pt}}

\numberwithin{equation}{section}
\declaretheorem[style=definition,sibling=equation]{definition}

\declaretheorem[style=remark,sibling=definition]{remark}

\declaretheorem[style=plain,sibling=definition]{theorem}
\declaretheorem[style=plain,sibling=definition]{lemma}
\declaretheorem[style=plain,sibling=definition]{proposition}
\declaretheorem[style=plain,sibling=definition]{corollary}
\declaretheorem[style=plain,sibling=definition]{conjecture}

\SetAlgoInsideSkip{medskip}
\SetAlCapSkip{\smallskipamount}
\SetArgSty{}

\newcommand{\N}{\mathbb{N}}
\newcommand{\Z}{\mathbb{Z}}
\newcommand{\Q}{\mathbb{Q}}
\newcommand{\R}{\mathbb{R}}

\DeclareMathOperator{\dgon}{dgon}
\DeclareMathOperator{\rank}{rank}
\DeclareMathOperator{\Div}{Div}
\DeclareMathOperator{\supp}{supp}
\DeclareMathOperator{\mdiv}{div} 
\DeclareMathOperator{\val}{val}

\newcommand{\pathcc}[2]{P_{[#1,#2]}} 
\newcommand{\pathco}[2]{P_{[#1,#2)}} 
\newcommand{\pathoc}[2]{P_{(#1,#2]}} 
\newcommand{\pathoo}[2]{P_{(#1,#2)}} 

\newcommand{\compl}[1]{#1^c}

\makeatletter
\DeclareMathOperator{\@Tm}{m}
\DeclareMathOperator{\@Tms}{ms}
\newcommand{\Tm}{\ensuremath{T_{\@Tm}}}
\newcommand{\Tms}{\ensuremath{T_{\@Tms}}}
\DeclareMathOperator{\@dgonr}{dgon}
\newcommand{\dgonr}[2]{\@dgonr_{#2}(#1)}
\DeclareMathOperator{\@Sym}{Sym}
\newcommand{\Sym}[1]{\@Sym(#1)}
\makeatother

\DeclareSymbolFont{bbold}{U}{bbold}{m}{n}
\DeclareSymbolFontAlphabet{\mathbbold}{bbold}
\newcommand{\one}{\ensuremath{\mathbbold{1}}}

\newcommand{\vv}[2]{v_{#1}^{#2}}

\begin{document}
\maketitle

\begin{abstract}
	This paper compares the divisorial gonality of a finite graph $G$ to the divisorial gonality of the associated metric graph $\Gamma(G,\one)$ with unit lengths.
	We show that $\dgon(\Gamma(G,\one))$ is equal to the minimal divisorial gonality of all regular subdivisions of $G$,
	and we provide a class of graphs for which this number is strictly smaller than the divisorial gonality of $G$.
	This settles a conjecture of M.~Baker \cite[Conjecture 3.14]{Baker-specialization} in the negative.
\end{abstract}

\section{Introduction}
Over the past 15 years, fruitful analogies between graphs and algebraic curves have been established, building on the seminal paper of Baker and Norine \cite{BN-Riemann-Roch}.
In that paper, the authors proved a version of the Riemann--Roch theorem for divisors on a finite graph, and showed that their result is closely related to the combinatorial theory of chip-firing games (e.g.{} \cite{BLS, Biggs}).
The paper of Baker and Norine, together with another paper by Baker \cite{Baker-specialization}, have led to a flurry of research into the interplay between graphs and curves, leading to new problems and results in combinatorics and to new combinatorial techniques in geometry (e.g.{} a combinatorial proof of the Brill--Noether theorem, \cite{CDPR}).

In \cite{Baker-specialization}, Baker posed a number of open problems in the theory of divisors on graphs.
All but two of these have since been solved; see \cite{Hladky-Kral-Norine, Luo, CDPR, Draisma-Vargas}.
The first and most important remaining open problem is the Brill--Noether conjecture for finite graphs \cite[Conj.~3.9(1)]{Baker-specialization}, based on an analogous result for curves.
We focus on the $r = 1$ case of this conjecture.
Let $\dgonr{G}{r}$ denote the smallest degree of a rank $r$ divisor on $G$, and let $\dgon(G) := \dgonr{G}{1}$ denote the \emph{divisorial gonality} of $G$ (see \mysecref{sec:prelims} for definitions).
Then the $r = 1$ case of the Brill--Noether conjecture can be stated as follows.
\begin{conjecture}[{Gonality conjecture, \cite[Conjecture 3.10(1)]{Baker-specialization}}]
	\label{conj:Brill-Noether}
	Let $G$ be a connected loopless multigraph, and let $g := |E(G)| - |V(G)| + 1$ denote its cyclomatic number.
	Then $\dgon(G) \leq \lfloor\frac{g + 3}{2}\rfloor$.
\end{conjecture}
The corresponding result for metric graphs was proven by Baker \cite[Thm.\ 3.12]{Baker-specialization} using algebraic geometry.
A purely combinatorial proof of this result was recently found by Draisma and Vargas \cite{Draisma-Vargas}, with many promising avenues still to be explored \cite{Draisma-Vargas-notices}.
However, for discrete graphs, \autoref{conj:Brill-Noether} is still wide open.\footnote{A proof of \autoref{conj:Brill-Noether} (and more generally \cite[Conj.~3.9(1)]{Baker-specialization}) was given by Caporaso \cite[Thm.~6.3]{Caporaso}, but a gap in this proof was later pointed out by Sam Payne and reported by Baker and Jensen in \cite[Rmk.~4.8 and footnote 5]{Baker-Jensen}. To our knowledge, this has not been repaired.}
Partial results were obtained by Atanasov and Ranganathan \cite{Atanasov-Ranganathan}, who proved \autoref{conj:Brill-Noether} for all graphs of genus at most $5$, and by Aidun and Morrison \cite{Aidun-Morrison}, who proved the conjecture for Cartesian product graphs.

The most straightforward approach to \autoref{conj:Brill-Noether} would be to show that the divisorial gonality of a graph is equal to the divisorial gonality of the associated metric graph with unit lengths.
This is the second remaining conjecture of Baker's paper \cite[Conj.~3.14]{Baker-specialization}.
Given a multigraph $G$ and an integer $k \geq 1$, let $\sigma_k(G)$ denote the multigraph obtained from $G$ by subdividing every edge into $k$ parts. The conjecture can then be stated as follows.

\begin{conjecture}[{\cite[Conjecture 3.14]{Baker-specialization}}]
	\label{conj:subdivision-metric}
	Let $G$ be a connected loopless multigraph, let $\Gamma(G)$ be the corresponding metric graph with unit edge lengths, and let $r \geq 1$. Then:
	\begin{enumerate}[label=(\alph*)]
		\item\label{itm:conj:subdivision} $\dgonr{G}{r} = \dgonr{\sigma_k(G)}{r}$ for all $k \geq 1$;
		\item\label{itm:conj:metric} $\dgonr{\Gamma(G)}{r} = \dgonr{G}{r}$.
	\end{enumerate}
\end{conjecture}

The first main result of this paper is that \myautoref{conj:subdivision-metric}{itm:conj:subdivision} and \myautoref{conj:subdivision-metric}{itm:conj:metric} are equivalent for every graph $G$.

\begin{theorem}
	\label{thm:two-conjectures}
	For every connected loopless multigraph $G$ and every integer $r \geq 1$, one has
	\[ \dgonr{\Gamma(G)}{r} = \min_{k\in\N_1} \dgonr{\sigma_k(G)}{r}. \]
\end{theorem}

It is already known that every rank $r$ divisor on $\sigma_k(G)$ also defines a rank $r$ divisor on $\Gamma(G)$. For the converse, we show that every rank $r$ divisor $D$ on $\Gamma(G)$ can be ``rounded'' to a nearby divisor $D'$ with $\rank(D') \geq r$ which is supported on the $\Q$-points of $\Gamma(G)$, and therefore on the points of some regular subdivision $\sigma_k(G)$. The details will be given in \mysecref{sec:two-conjectures}.

As pointed out by Baker in \cite{Baker-specialization}, a positive answer to \autoref{conj:subdivision-metric} would also yield a positive answer to \autoref{conj:Brill-Noether}.
However, it turns out that the subdivision conjecture fails, and we give a counterexample to \myautoref{conj:subdivision-metric}{itm:conj:subdivision} in the case $r = 1$ and $k = 2$.
Evidently this is also a counterexample to \myautoref{conj:subdivision-metric}{itm:conj:metric}.
The second main result of this paper is the following.

\begin{theorem}
	\label{thm:counterexample}
	For every integer $k \geq 1$, there exists a connected loopless multigraph $G_k$ such that $\dgon(G_k) = 6k$ and $\dgon(\Gamma(G_k)) = \dgon(\sigma_2(G_k)) = 5k$. Furthermore, $G_k$ can be chosen simple and bipartite.
\end{theorem}

The proof is constructive and consists of two parts.
In \mysecref{sec:main-counterexample}, we construct a family of graphs with $\dgon(G) = 6$ and $\dgon(\Gamma(G)) = \dgon(\sigma_2(G)) = 5$.
The graphs $G_k$ are then constructed in \mysecref{sec:skewered} by combining $k$ of these graphs in a certain way.

Although the difference between $\dgon(G)$ and $\dgon(\Gamma(G))$ can be large, as in \autoref{thm:counterexample}, the ratio between them is at most $2$, as we show in \autoref{prop:bound}.
Hence, for the gap to get arbitrarily large, it is necessary that $\dgon(\Gamma(G))$ goes to infinity.

In \mysecref{sec:closing-remarks}, we list a few additional counterexamples (without proof), including a $3$-regular graph.
Although all counterexamples in this paper violate \autoref{conj:subdivision-metric}, they nevertheless satisfy the Brill--Noether bound.
We do not know whether any of these examples can be extended to disprove \autoref{conj:Brill-Noether}.
Additional open problems are discussed in \mysecref{sec:closing-remarks} as well.

\section{Preliminaries}
\label{sec:prelims}
Throughout this paper, a \emph{graph} will be a finite, connected, loopless multigraph. In other words, parallel edges are allowed, but self-loops are not.

\subsection{Divisors on graphs}

A \emph{divisor} on a graph $G$ is an element of the free abelian group on $G$.
In other words, a divisor is a formal sum $\sum_{v \in V(G)} a_v v$, where $a_v \in \Z$ for all $v$.
If $D$ is a divisor on $G$ and if $w \in V(G)$, then we use the notation $D(w)$ to denote the coefficient $a_w$ of $w$ in $D$.
The \emph{support} $\supp(D)$ of a divisor $D$ is the set of all $v$ for which $D(v) \neq 0$.

For two divisors $D$ and $D'$, we write $D \geq D'$ if $D(v) \geq D'(v)$ for all $v$.
A divisor $D$ is called \emph{effective} if $D \geq 0$.
The sets of all divisors and all effective divisors on $G$ are denoted by $\Div(G)$ and $\Div_+(G)$, respectively.

The \emph{degree} of a divisor is the sum of its coefficients: $\deg(D) := \sum_{v\in V(G)} D(v)$.
The set of all effective divisors of degree $d$ on $G$ is denoted $\Div_+^d(G)$.

The Laplacian matrix $L_G$ of $G$ defines a map $\Z^{V(G)} \to \Div(G)$, $x \mapsto L_Gx$.
Divisors in the image of this map are called \emph{principal divisors}.
Two divisors $D,D' \in \Div(G)$ are \emph{equivalent} if $D - D'$ is a principal divisor.

Equivalence of divisors can also be described in terms of the following chip-firing game.
An effective divisor $D$ is interpreted as a distribution of chips over the vertices of $G$, where $D(v)$ is the number of chips on $v$.
A subset $A \subseteq V(G)$ is \emph{valid} (with respect to $D$) if $D(v) \geq |\{uv \in E(G) \mid u\notin A\}|$ for all $v\in A$.
If $A$ is valid, then to \emph{fire $A$} is to move chips from $A$ to $V(G) \setminus A$, one for every edge of the cut $(A, V(G) \setminus A)$.
This yields a new divisor $D'$ given by $D' = D - L_G\one_A$, where $\one_A$ is the characteristic vector of $A$.
Since $A$ is valid, this new divisor $D'$ is again effective.
All equivalent effective divisors can be reached in this way:

\begin{proposition}[{\cite[Lem.~2.3]{JosseDionTreewidth}}]
	\label{prop:increasing-firing-sets}
	Let $D,D'$ be equivalent effective divisors. Then $D'$ can be obtained from $D$ by subsequently firing an increasing sequence of valid subsets
	\[ \varnothing \subsetneq U_1 \subseteq U_2 \subseteq \cdots \subseteq U_k \subsetneq V. \]
\end{proposition}

\noindent
The \emph{rank} of a divisor $D \in \Div(G)$ is defined as
\[ \rank(D) := \max\{k \in \Z \mid D - E \ \text{is equivalent to an effective divisor for all $E \in \Div_+^k(G)$}\}. \]
We have $\rank(D) = -1$ if and only if $D$ is not equivalent to an effective divisor.

Given a graph $G$ and an integer $r \geq 1$, the \emph{$r$-th \textup(divisorial\textup) gonality} $\dgonr{G}{r}$ of $G$ is the minimum degree of a rank $r$ divisor on $G$. For $r = 1$, this is simply called the \emph{\textup(divisorial\textup) gonality} of $G$: $\dgon(G) := \dgonr{G}{1}$.

\subsection{Reduced divisors and Dhar's burning algorithm}

Let $G$ be a graph, and let $q \in V(G)$.
A divisor $D \in \Div(G)$ is called \emph{$q$-reduced} if $D(v) \geq 0$ for all $v \in V(G) \setminus \{q\}$ and every non-empty valid set contains $q$.
Every divisor $D$ is equivalent to a unique $q$-reduced divisor; see \cite[Prop.~3.1]{BN-Riemann-Roch}.
A divisor $D$ has rank at least 1 if and only if for every vertex $v$, the $v$-reduced divisor $D_v$ equivalent to $D$ has at least one chip on $v$.

Dhar's burning algorithm \cite{Dhar} takes as input a graph $G$, a divisor $D$ and a vertex $q$, and returns the maximal valid set $A \subseteq V(G) \setminus \{q\}$. In particular, the set $A$ returned by Dhar's burning algorithm is empty if and only if $D$ is $q$-reduced.

\begin{algorithm}[ht]
	\SetKwInOut{KwIn}{Input}
	\SetKwInOut{KwOut}{Output}
	\KwIn{A triple $(G,D,q)$, where $G$ is a graph, $D \in \Div_+(G)$, and $q \in V(G)$.}
	\KwOut{The maximal valid subset $A \subseteq V(G) \setminus \{q\}$.}
	\BlankLine
	Burn vertex $q$\;
	Burn all edges incident to burned vertices\;\label{line:Dhar2}
	If a vertex $v$ is incident to more burned edges than it has chips, burn $v$\;\label{line:Dhar3}
	Repeat steps \ref{line:Dhar2} and \ref{line:Dhar3} until no more edges or vertices are burned\;
	\KwRet{$\{v \in V(G) \mid v\ \text{is not burned}\}$}
	\caption{Dhar's burning algorithm for finite graphs}
	\label{alg:Dhar}
\end{algorithm}

\subsection{Metric graphs}

A \emph{metric graph} is a metric space $\Gamma$ that can be obtained in the following way.
Let $G$ be a finite multigraph and let $\ell\colon E(G) \to \R_{>0}$ be an assignment of lengths to the edges of $G$.
To construct $\Gamma$, take an interval $[0, \ell(e)]$ for every edge $e\in E(G)$, and glue these together at the endpoints as prescribed by $G$.
To turn it into a metric space, equip $\Gamma$ with the shortest path metric in the obvious way.
The metric graph $\Gamma$ defined in this way will be denoted $\Gamma(G,\ell)$.
If $\ell = \one$ is the unit length function, we write $\Gamma(G) := \Gamma(G,\one)$.

If the metric graph $\Gamma$ is constructed from the pair $(G,\ell)$ as above, then we say that $(G,\ell)$ is a \emph{model} of $\Gamma$. We say that a model $(G,\ell)$ is \emph{loopless} (resp.{} \emph{simple}) if $G$ is loopless (resp.{} simple). The \emph{valency} $\val(v)$ of $v \in \Gamma$ is the number of edges incident with $v$ in any loopless model $(G,\ell)$ with $v \in V(G)$.

A \emph{divisor} on a metric graph $\Gamma$ is an element of the free abelian group on $\Gamma$.
In other words, a divisor is a formal sum $\sum_{v \in \Gamma} a_v v$ where $a_v \in \Z$ for all $v$, and $a_v = 0$ for all but finitely many $v$.
The notations $\supp(D)$, $\deg(D)$, $D \geq D'$, $\Div(\Gamma)$, $\Div_+(\Gamma)$ and $\Div_+^d(\Gamma)$ are defined analogously to the discrete case.

The definition of equivalence is a bit different.
A \emph{rational function} on $\Gamma$ is a continuous piecewise linear function $f\colon \Gamma \to \R$ with integral slopes.
For each point $v \in \Gamma$, let $a_v$ be the sum of the outgoing slopes of $f$ in all edges incident to $v$ in some appropriate model of $\Gamma$.
The corresponding divisor $\sum_{v\in \Gamma} a_v v$ is called a \emph{principal divisor}.
Two divisors $D$ and $D'$ are \emph{equivalent} if $D - D'$ is a principal divisor.

The \emph{rank} of a divisor $D \in \Div(\Gamma)$ is defined as in the discrete case; that is:
\[ \rank(D) := \max\{k \in \Z \mid D - E \ \text{is equivalent to an effective divisor for all $E \in \Div_+^k(\Gamma)$}\}. \]
The \emph{$r$-th \textup(divisorial\textup) gonality} $\dgonr{\Gamma}{r}$ of $\Gamma$ is the minimum degree of a rank $r$ divisor on $\Gamma$. For $r = 1$, this is simply called the \emph{\textup(divisorial\textup) gonality} of $\Gamma$: $\dgon(\Gamma) := \dgonr{\Gamma}{1}$.

If $G$ is a finite graph and if $\Gamma := \Gamma(G)$ is the corresponding metric graph with unit lengths, then two divisors $D,D' \in \Div(G)$ are equivalent on $G$ if and only if they are equivalent on $\Gamma$; see \cite[Rmk.~1.3]{Baker-specialization}. Furthermore, in this case one has $\rank_G(D) = \rank_\Gamma(D)$ for every divisor $D \in \Div(G)$; see \cite[Thm.~1.3]{Hladky-Kral-Norine}.

\subsection{Rank-determining sets and strong separators}

Let $\Gamma$ be a metric graph, and let $S \subseteq \Gamma$ be a subset.
Following \cite{Luo}, we define the \emph{$S$-restricted rank} of a divisor $D \in \Div(\Gamma)$ as
\[ \rank_S(D) := \max\{k \mid D - E \ \text{is equivalent to an effective divisor for all $E \in \Div_+^k(S)$}\}, \]
where $\Div_+^k(S)$ is the set of degree $k$ effective divisors whose support is contained in $S$.
The set $S$ is \emph{rank-determining} if $\rank_S(D) = \rank(D)$ for all $D \in \Div(\Gamma)$.
The following theorems are due to Luo.

\begin{theorem}[{\cite[Thm.~1.6]{Luo}; see also \cite[Thm.~1.7]{Hladky-Kral-Norine}}]
	\label{thm:rank-determining-set}
	Let $\Gamma$ be a metric graph, and let $(G,\ell)$ be a loopless model of $\Gamma$.
	Then the set $V(G) \subseteq \Gamma$ is rank-determining.
\end{theorem}
\begin{theorem}[{\cite[Thm.~1.10]{Luo}}]
	\label{thm:rank-determining-homeomorphism}
	Let $\Gamma,\Gamma'$ be metric graphs, and let $\phi\colon \Gamma \to \Gamma'$ be a homeomorphism. Then $A \subseteq \Gamma$ is rank-determining if and only if $\phi[A] \subseteq \Gamma'$ is rank-determining.
\end{theorem}

We also formulate a discrete analogue of \autoref{thm:rank-determining-set} for the case $r = 1$.
If $G$ is a graph, then we say that a divisor $D \in \Div(G)$ \emph{reaches} the vertex $v \in V(G)$ if there is an effective divisor $D'$ equivalent to $D$ with $D'(v) > 0$.
Furthermore, we say that a set $S \subseteq V(G)$ is a \emph{strong separator} if for every connected component $C$ of $V(G) \setminus S$ we have that $C$ is a tree and for every $s \in S$ there is at most one edge (in $G$) between $C$ and $s$.

\begin{theorem}[{\cite[Lem.~2.6]{JosseDionTreewidth}}]
	\label{thm:strong-separator}
	Let $G$ be a graph, and let $S \subseteq V(G)$ be a strong separator. If $D \in \Div(G)$ reaches every $s \in S$, then $\rank(D) \geq 1$.
\end{theorem}

The following corollary is immediate from either \autoref{thm:rank-determining-set} or \autoref{thm:strong-separator}.

\begin{corollary}
	\label{cor:discrete-rank-determining-set}
	Let $G$ be a loopless multigraph, and let $H$ be a subdivision of $G$. If $D \in \Div(H)$ reaches all vertices of $V(G)$, then $\rank(D) \geq 1$.
\end{corollary}

\section{Equivalence of the two forms of \texorpdfstring{\autoref{conj:subdivision-metric}}{Conjecture \ref*{conj:subdivision-metric}}}
\label{sec:two-conjectures}

In this section, we prove \autoref{thm:two-conjectures} using a modification of the proof of \cite[Thm.~5.1]{JosseDionTreewidth}.
The main idea is the following: given a rank $r$ divisor $D$ on the metric graph $\Gamma(G)$, we will change the lengths of the edges between points in $V(G) \cup \supp(D)$ in such a way that $\supp(D)$ is moved to the $\Q$-points of the graph, all the while leaving the rank of $D$ and the distances between the vertices of $G$ unchanged.
We will now make this precise.

\begin{definition}
	Given a metric graph $\Gamma$ and a model $(G,\ell)$ of $\Gamma$, a \emph{$G$-rescaling of $\Gamma$} is a metric graph $\Gamma' := \Gamma(G,\ell')$, where $\ell' \in \R_{>0}^{E(G)}$ is another length vector.
	If $D \in \Div(\Gamma)$ with $\supp(D) \subseteq V(G)$, then $D$ defines a divisor $D' \in \Div(\Gamma')$ in the obvious way, which we call \emph{the $G$-rescaling of $D$}.
\end{definition}
We point out that $\Gamma$ and its $G$-rescaling $\Gamma'$ can be isometric even if $\ell \neq \ell'$. This is because vertices of degree $2$ can be moved around, as illustrated in \autoref{fig:isometric-rescaling}.
In that case the vertex set $V(G)$ is embedded into $\Gamma \cong \Gamma'$ in two different ways, and the divisor $D$ and its $G$-rescaling $D'$ could be different divisors on the same metric graph.
This will be the main tool in our proof of \autoref{thm:two-conjectures}.

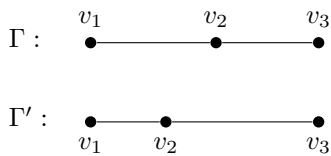
\begin{figure}[h!t]
	\centering
	\begin{tikzpicture}[scale=3]
		\node[vertex,label=above:$v_1$] (v1) at (0,0) {};
		\node[vertex,label=above:$v_2$] (v2) at (0.55,0) {};
		\node[vertex,label=above:$v_3$] (v3) at (1,0) {};
		\draw (v1) -- (v2) -- (v3);
		\node[anchor=base west] at (-.4,-.02) {$\Gamma:$};
		\begin{scope}[yshift=-3.5mm]
			\node[vertex,label=below:$v_1$] (w1) at (0,0) {};
			\node[vertex,label=below:$v_2$] (w2) at (0.33,0) {};
			\node[vertex,label=below:$v_3$] (w3) at (1,0) {};
			\draw (w1) -- (w2) -- (w3);
			\node[anchor=base west] at (-.4,-.02) {$\Gamma':$};
		\end{scope}
	\end{tikzpicture}
	\caption{A metric graph $\Gamma = \Gamma(G,\ell)$ and a rescaling $\Gamma' = \Gamma(G,\ell')$ with $\ell' \neq \ell$ such that $\Gamma$ and $\Gamma'$ are isometric.}
	\label{fig:isometric-rescaling}
\end{figure}

To rescale from real to rational edge lengths we use the following lemma.

\begin{lemma}
	\label{lem:rational-LP}
	Let $A \in \Q^{m \times n}$ and $b \in \Q^m$. If the linear system $Ax = b$ has a solution $x \in \R_{>0}^n$, then it also has a solution $x' \in \Q_{>0}^n$.
\end{lemma}
\begin{proof}
	Since the system has a solution $x \in \R_{>0}$, the solution space $\{z \mid Az = b\}$ is a non-empty affine $\Q$-subspace of $\R^n$. Choose an affine rational basis $y_0,\ldots,y_d \in \Q^n$ for the solution space and write $x = \alpha^{(0)} y_0 + \cdots + \alpha^{(d)} y_d$ with $\alpha^{(0)} + \cdots + \alpha^{(d)} = 1$. For every $i \in \{1,\ldots,d\}$, choose a rational sequence $\{\alpha_k^{(i)}\}_{k=1}^\infty$ such that $\lim_{k \to \infty} \alpha_k^{(i)} = \alpha^{(i)}$, and define $\alpha_k^{(0)} := 1 - \alpha_k^{(1)} - \cdots - \alpha_k^{(d)}$. Then $\lim_{k\to\infty} \alpha_k^{(0)}y_0 + \cdots + \alpha_k^{(d)}y_d = x$. Since $\R_{>0}^n$ is an open neighbourhood of $x$, there is a $K_0 \in \N$ such that $\alpha_k^{(0)}y_0 + \cdots + \alpha_k^{(d)}y_d \in \R_{>0}^n$ for all $k \geq K_0$. This gives a sequence of solutions in $\Q_{>0}^n$ converging to $x$.
\end{proof}

We now come to the main result of this section, which is an extension of \cite[Thm.~5.1]{JosseDionTreewidth}. The result of \myautoref{thm:rounding}{itm:rounding:real} was already implicit in the proof of \cite[Prop.~3.1]{Gathmann-Kerber}.

\begin{theorem}
	\label{thm:rounding}
	Let $\Gamma$ be a metric graph, and let $D \in \Div_+(\Gamma)$ be an effective divisor.
	\begin{enumerate}[label=(\alph*)]
		\item\label{itm:rounding:real} There exists a loopless model $(G,\ell)$ with $\supp(D) \subseteq V(G)$ and a rational length vector $\ell' \in \Q_{>0}^{E(G)}$ such that the $G$-rescaling $D'$ of $D$ in $\Gamma' := \Gamma(G,\ell')$ satisfies $\rank_{\Gamma'}(D') \geq \rank_{\Gamma}(D)$.
		
		\item\label{itm:rounding:rational} If $\Gamma$ is a metric $\Q$-graph, then the length vector $\ell'$ in \ref{itm:rounding:real} can be chosen in such a way that $\Gamma'$ is isometric to $\Gamma$.
	\end{enumerate}
\end{theorem}
\begin{proof}
	\begin{enumerate}[label=(\alph*)]
		\item Write $r := \rank_\Gamma(D)$, and let $S \subseteq \Gamma$ be a finite rank-determining set.
		For every $E \in \Div_+^r(S)$, choose a divisor $D_E \in \Div(\Gamma)$ and a rational function $f_E\colon \Gamma \to \R$ such that $D_E \geq E$ and $D_E = D + \mdiv(f_E)$.
		Furthermore, choose a loopless model $(G,\ell)$ of $\Gamma$ such that
		\[ S \cup \supp(D) \cup \bigcup_{E \in \Div_+^r(S)} \supp(D_E) \, \subseteq \, V(G). \]
		Since $D,D_E \geq 0$ and $D_E - D = \mdiv(f_E)$, we have $\supp(\mdiv(f_E)) \subseteq \supp(D) \cup \supp(D_E) \subseteq V(G)$, so $V(G)$ contains all points of non-linearity of $f_E$, for every $E \in \Div_+^r(S)$.
		
		Choose an orientation of the edges of $G$. For every cycle $C$ in $G$, choose a circular orientation of the edges of $C$, and define $\chi_C\colon E(G) \to \{-1,0,1\}$ by setting
		\[ \chi_C(e) := \begin{cases}
			1,&\quad\text{if $e \in E(C)$ and the orientations of $G$ and $C$ agree on $e$};\\[1ex]
			-1,&\quad\text{if $e \in E(C)$ and the orientations of $G$ and $C$ disagree on $e$};\\[1ex]
			0,&\quad\text{if $e \notin E(C)$}.
		\end{cases} \]
		For $E \in \Div_+^r(S)$ and $e \in E(G)$, let $\phi(f_E,e) \in \Z$ denote the slope of $f_E$ on $e$, in the forward direction of $e$.
		Note that a $G$-rescaling $\Gamma(G,\ell')$ of $\Gamma$ admits rational functions $f_E'$ whose slope on $e$ equals $\phi(f_E,e)$, for all $E \in \Div_+^r(S)$ and all $e\in E(G)$, if and only if $y = \ell'$ is a solution to following system of equations:
		\begin{equation}
			\sum_{e\in E(G)} \phi(f_E,e)\chi_C(e)\, y(e) = 0,\ \text{for every cycle $C$ and every $E \in \Div_+^r(S)$}. \label{eqn:rounding:linear-system}
		\end{equation}
		Since the coefficients (that is, $\phi(f_E,e)\chi_C(e)$) and constants (that is, $0$) of this linear system are integral, and since $y = \ell \in \R_{>0}^{E(G)}$ is a solution, it follows from \autoref{lem:rational-LP} that there exists a solution $\ell' \in \Q_{>0}^{E(G)}$.

		Consider the $G$-rescaling $\Gamma' := \Gamma(G,\ell')$. Let $D'$ be the corresponding $G$-rescaling of $D$, and let $D_E'$ be the $G$-rescaling of $D_E$ for all $E \in \Div_+^r(S)$.
		By the above, we may choose rational functions $f_E'$ on $\Gamma'$ such that the slope of $f_E'$ on $e$ equals $\phi(f_E,e)$, for all $E \in \Div_+^r(S)$ and all $e \in E(G)$.
		Then clearly $D_E' = D' + \mdiv(f_E')$, so the $D_E'$ are equivalent to $D'$.
		Since $D_E' \geq E$ for every $E \in \Div_+^r(S)$, it follows that $\rank_S(D') \geq r$.
		By \autoref{thm:rank-determining-homeomorphism}, $S$ is rank-determining in $\Gamma'$, so $\rank_{\Gamma'}(D') = \rank_S(D') \geq r$.
		
		\item Choose a rational model $(\tilde G,\tilde\ell)$ of $\Gamma$.
		We repeat the argument of \ref{itm:rounding:real} with the following modifications.
		First, we add the requirement that $V(\tilde G) \subseteq V(G)$.
		Then every edge $\tilde e$ in $\tilde G$ corresponds to a path in $G$, which we denote $P_{\tilde e}$.
		Second, we extend the linear system from \eqref{eqn:rounding:linear-system} by adding the following constraints:
		\begin{equation}
			\sum_{e \in E(P_{\tilde e})} y(e) = \tilde \ell(\tilde e),\quad\text{for all $\tilde e \in E(\tilde G)$}. \label{eqn:rounding:linear-system-ii}
		\end{equation}
		Again, the coefficients and constants of the linear system are rational, and $y = \ell \in \R_{>0}^{E(G)}$ is a solution, so it follows from \autoref{lem:rational-LP} that there is a solution $\ell' \in \Q_{>0}^{E(G)}$. The rest of the proof of \ref{itm:rounding:real} carries through unchanged, and the extra constraints from \eqref{eqn:rounding:linear-system-ii} ensure that $\Gamma'$ is isometric to $\Gamma$.
		\qedhere
	\end{enumerate}
\end{proof}

\begin{proof}[{Proof of \autoref{thm:two-conjectures}}]
	Every rank $r$ divisor on $\sigma_k(G)$ also defines a rank $r$ divisor on $\Gamma(\sigma_k(G),\one/k) = \Gamma(G,\one)$. Therefore $\dgonr{\Gamma(G)}{r} \leq \min_{k\in\N_1} \dgonr{\sigma_k(G)}{r}$.
	
	Conversely, let $D \in \Div_+(\Gamma(G))$ be an effective divisor of rank $r$. By \myautoref{thm:rounding}{itm:rounding:rational}, there exists a divisor $D' \in \Div_+(\Gamma(G))$ with $\deg(D') = \deg(D)$ and $\rank(D') \geq \rank(D)$ which is supported on the $\Q$-points of $\Gamma(G)$. Then $D'$ is supported on the vertices of the model $(\sigma_k(G),\one/k)$ of $\Gamma(G)$ for some $k \in \N_1$, so we have $\dgonr{\sigma_k(G)}{r} \leq \dgonr{\Gamma(G)}{r}$.
\end{proof}

\begin{remark}
	Analogously to the proof of the main result of \cite{Bodlaender-vdWegen-vdZanden}, the linear system from the proof of \myautoref{thm:rounding}{itm:rounding:rational} forms a certificate that $\dgonr{\Gamma}{r} \leq d$.
	If $r$ and $d$ are fixed, then this certificate has size polynomial in the size of $\Gamma$, so it follows that \textsc{Metric Divisorial $r$-Gonality} for $\Q$-graphs belongs to the complexity class NP.
	(For details, refer to the proof in \cite{Bodlaender-vdWegen-vdZanden}.)
	Moreover, it can be deduced from the proof of \cite[Thm.~3.5]{Gijswijt-Smit-vdWegen} that this problem is also NP-hard for $r = 1$.
	We suspect that the same holds for all $r \geq 1$.
\end{remark}

\begin{remark}
	The proof of \myautoref{thm:rounding}{itm:rounding:rational} can also be used to find an upper bound on the size of the subdivision needed to get equality in \autoref{thm:two-conjectures}.
	One such upper bound can be obtained by following the proof of \cite[Cor.~6.2]{Bodlaender-vdWegen-vdZanden}.
	We sketch a way to improve this bound.
	Let $\tilde G$ be a graph with $n$ vertices and $m$ edges, let $\Gamma := \Gamma(\tilde G)$ be the corresponding unit metric graph, and let $D \in \Div(\Gamma)$ be a divisor of degree $d$ and rank $r$.
	We repeat the proof of \myautoref{thm:rounding}{itm:rounding:rational} with respect to the rational model $(\tilde G,\one)$ and the rank-determining set $S := V(\tilde G)$ (use \autoref{thm:rank-determining-set}).
	Without loss of generality, we may assume that $D$ is equal to one of the $D_E$.
	Then the number of variables of the linear system is $|E(G)| \leq m + dn^r$.
	
	Note that we can also allow a solution $\ell' \geq 0$ instead of $\ell' > 0$. This has the effect of contracting some of the edges of the model $G$ from the proof of \autoref{thm:rounding}, but the equations from \eqref{eqn:rounding:linear-system-ii} ensure that the resulting graph $\Gamma'$ is still isometric to $\Gamma$.
	Hence \eqref{eqn:rounding:linear-system} and \eqref{eqn:rounding:linear-system-ii} determine a linear program $Ax = b$, $x \geq 0$, and the entries of $A$ are integers which can be shown to be bounded in absolute value by $d$.
	The set of feasible solutions is non-empty and bounded by \eqref{eqn:rounding:linear-system-ii}, so there is a basic feasible solution $x$ (see e.g.~\cite[Thm.~4.2.3]{Matousek}).
	Hence there is a subset $B \subseteq \{1,\ldots,|E(G)|\}$ such that $x_B^{} = A_B^{-1}b$ and $x_{\compl{B}}^{} = 0$.
	Therefore the lowest common denominator of the entries of $x$ is at most $|\det(A_B)| \leq \sum_{\sigma \in \Sym{B}} \prod_{i \in B} |A_{i\sigma(i)}| \leq |B|!\cdot d^{|B|}$.
	In conclusion, if the unit metric graph $\Gamma = \Gamma(\tilde G,\one)$ has a divisor of rank $r$ and degree $d$, then so does $\sigma_k(\tilde G)$ for some $k \leq (m + dn^r)! \cdot d^{m + dn^r}$.
\end{remark}

\section{A graph \texorpdfstring{$G$}{G} such that \texorpdfstring{$\dgon(\sigma_2(G)) < \dgon(G)$}{dgon(σ\textunderscore{}2(G)) < dgon(G)}}
\label{sec:main-counterexample}

In this section we construct a class of graphs, which we call ``tricycle graphs''. We show that the divisorial gonality of any tricycle graph $G$ is strictly greater than the divisorial gonality of its $2$-subdivision $\sigma_2(G)$, and thus of its associated metric graph $\Gamma(G)$.

\begin{definition}
	\label{def:tricycle}
	A \emph{tricycle graph} is a multigraph $G$ that can be obtained in the following way:
	\begin{itemize}
		\item Start with three disjoint cycles $C_1$, $C_2$, $C_3$, each on at least $2$ vertices (a cycle on $2$ vertices consists of two vertices connected by two parallel edges).
		\item Choose two distinct vertices on each of these cycles, say $\vv{i}{-},\vv{i}{+} \in V(C_i)$.
		\item Add the edges $\vv{1}{+}\vv{2}{-}$, $\vv{2}{+}\vv{3}{-}$ and $\vv{3}{+}\vv{1}{-}$ to the graph.
		\item Add another vertex $v_0$ and six new edges to $G$, connecting $v_0$ to the six vertices $\vv{1}{-},\vv{1}{+},\vv{2}{-},\vv{2}{+},\vv{3}{-},\vv{3}{+}$.
		\item Subdivide the six edges incident with $v_0$.
	\end{itemize}
	We call the vertices $\vv{i}{-},\vv{i}{+} \in V(C_i)$ the \emph{transition vertices}, the edges $\vv{1}{+}\vv{2}{-}$, $\vv{2}{+}\vv{3}{-}$ and $\vv{3}{+}\vv{1}{-}$ the \emph{transition edges}, and $v_0$ the \emph{central vertex}.
	The \emph{outer ring} is the union of the cycles $C_1,C_2,C_3$ and the transition edges.
\end{definition}

\autoref{fig:tricycle} illustrates an example of a tricycle graph, along with the minimal tricycle $\Tm$ and the minimal simple tricycle $\Tms$.
Note that a multigraph (resp.~simple graph) $G$ is a tricycle if and only if $G$ can be obtained by taking a subdivision of the minimal tricycle $\Tm$ (resp.~the minimal simple tricycle $\Tms$) in such a way that the transition edges are not subdivided.

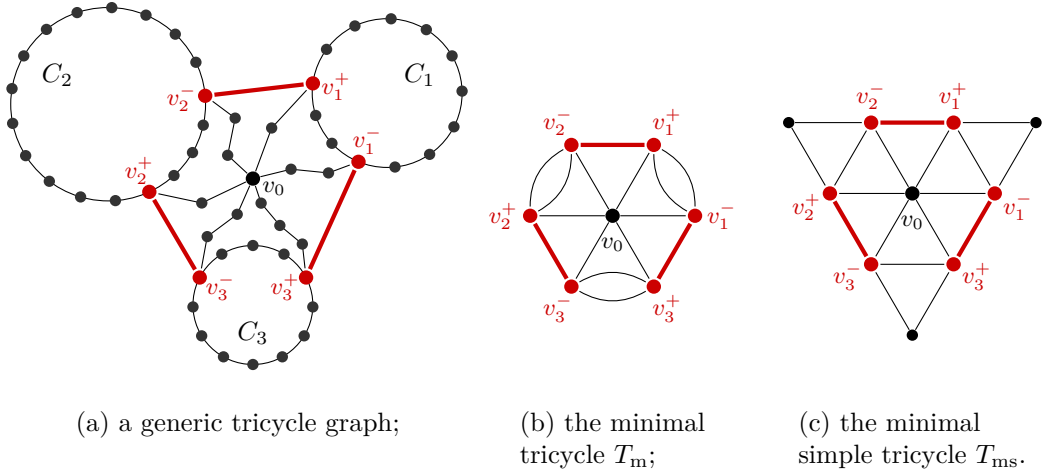
\begin{figure}[h!bt]
	\centering
	\def\cykelhoekA{75}
	\def\cykelafstA{23}
	\def\cykelradiusA{10}
	\def\cykelhoekB{70}
	\def\cykelafstB{20}
	\def\cykelradiusB{13}
	\def\cykelhoekC{125}
	\def\cykelafstC{17}
	\def\cykelradiusC{8}
	\pgfmathsetmacro{\curschaal}{\schaal * 0.58}
	\begin{tikzpicture}[scale=\curschaal,gray_vertex/.style={vertex,black!80},transition/.style={red!80!black,ultra thick,larger},larger/.style={inner sep=1.9pt}]
		\begin{scope}
			\coordinate (C1) at (30:\cykelafstA mm);
			\coordinate (C2) at (150:\cykelafstB mm);
			\draw (270:\cykelafstC mm) ++(.2,0) coordinate (C3);
			\draw (C1) circle (\cykelradiusA mm);
			\draw (C2) circle (\cykelradiusB mm);
			\draw (C3) circle (\cykelradiusC mm);
			\pgfmathsetmacro{\hoekA}{210 + \cykelhoekA / 2}
			\pgfmathsetmacro{\hoekB}{210 - \cykelhoekA / 2}
			\pgfmathsetmacro{\hoekC}{-30 + \cykelhoekB / 2}
			\pgfmathsetmacro{\hoekD}{-30 - \cykelhoekB / 2}
			\pgfmathsetmacro{\hoekE}{90 + \cykelhoekC / 2}
			\pgfmathsetmacro{\hoekF}{90 - \cykelhoekC / 2}
			\node[vertex,larger] (v0) at (.2,0) {};
			\draw (v0) ++(-25:.31) node {\small$v_0$};
			\begin{scope}[shift=(C1)]
				\node[vertex,transition] (v1) at (\hoekA:\cykelradiusA mm) {};
				\node[vertex,transition] (v2) at (\hoekB:\cykelradiusA mm) {};
				\pgfmathsetmacro{\labelradiusA}{\cykelradiusA - 3.14}
				\node[transition] at (\hoekA:\labelradiusA mm) {\small$\vv{1}{-}$};
				\node[transition] at (\hoekB:\labelradiusA mm) {\small$\vv{1}{+}$};
				\def\aantalkort{2}
				\foreach \x in {1,...,\aantalkort} {
					\pgfmathsetmacro{\curhoek}{\hoekB + \x * (\hoekA - \hoekB) / (\aantalkort + 1)}
					\node[gray_vertex] at (\curhoek:\cykelradiusA mm) {};
				}
				\def\aantallang{9}
				\foreach \x in {1,...,\aantallang} {
					\pgfmathsetmacro{\curhoek}{\hoekA + \x * (\hoekB - \hoekA + 360) / (\aantallang + 1)}
					\node[gray_vertex] at (\curhoek:\cykelradiusA mm) {};
				}
			\end{scope}
			\begin{scope}[shift=(C2)]
				\node[vertex,transition] (v3) at (\hoekC:\cykelradiusB mm) {};
				\node[vertex,transition] (v4) at (\hoekD:\cykelradiusB mm) {};
				\pgfmathsetmacro{\labelradiusB}{\cykelradiusB - 3.14}
				\node[transition] at (\hoekC:\labelradiusB mm) {\small$\vv{2}{-}$};
				\node[transition] at (\hoekD:\labelradiusB mm) {\small$\vv{2}{+}$};
				\def\aantalkort{3}
				\foreach \x in {1,...,\aantalkort} {
					\pgfmathsetmacro{\curhoek}{\hoekD + \x * (\hoekC - \hoekD) / (\aantalkort + 1)}
					\node[gray_vertex] at (\curhoek:\cykelradiusB mm) {};
				}
				\def\aantallang{13}
				\foreach \x in {1,...,\aantallang} {
					\pgfmathsetmacro{\curhoek}{\hoekC + \x * (\hoekD - \hoekC + 360) / (\aantallang + 1)}
					\node[gray_vertex] at (\curhoek:\cykelradiusB mm) {};
				}
			\end{scope}
			\begin{scope}[shift=(C3)]
				\node[vertex,transition] (v5) at (\hoekE:\cykelradiusC mm) {};
				\node[vertex,transition] (v6) at (\hoekF:\cykelradiusC mm) {};
				\pgfmathsetmacro{\labelradiusC}{\cykelradiusC - 3.14}
				\node[anchor=base,yshift=-2.2pt,transition] at (\hoekE:\labelradiusC mm) {\small$\vv{3}{-}$};
				\node[anchor=base,yshift=-2.2pt,transition] at (\hoekF:\labelradiusC mm) {\small$\vv{3}{+}$};
				\def\aantalkort{3}
				\foreach \x in {1,...,\aantalkort} {
					\pgfmathsetmacro{\curhoek}{\hoekF + \x * (\hoekE - \hoekF) / (\aantalkort + 1)}
					\node[gray_vertex] at (\curhoek:\cykelradiusC mm) {};
				}
				\def\aantallang{7}
				\foreach \x in {1,...,\aantallang} {
					\pgfmathsetmacro{\curhoek}{\hoekE + \x * (\hoekF - \hoekE + 360) / (\aantallang + 1)}
					\node[gray_vertex] at (\curhoek:\cykelradiusC mm) {};
				}
			\end{scope}
			\path (C1) to node {$C_1$} ++(30:\cykelradiusA mm);
			\path (C2) to node[pos=.6] {$C_2$} ++(150:\cykelradiusB mm);
			\path (C3) to node[pos=.47] {$C_3$} ++(270:\cykelradiusC mm);
			\draw[transition] (v2) -- (v3);
			\draw[transition] (v4) -- (v5);
			\draw[transition] (v6) -- (v1);
			\draw (v1) -- ++(-.4,-.15) node[gray_vertex] {} -- ++(-.5,.05) node[gray_vertex] {} -- (v0);
			\draw (v2) -- ++(-.55,-.6) node[gray_vertex] {} -- (v0);
			\draw (v3) -- ++(.4,-.3) node[gray_vertex] {} -- ++(-.1,-.5) node[gray_vertex] {} -- (v0);
			\draw (v4) -- ++(.7,-.15) node[gray_vertex] {} -- (v0);
			\draw (v5) -- ++(.1,.55) node[gray_vertex] {} -- ++(.4,.3) node[gray_vertex] {} -- (v0);
			\draw (v6) -- ++(-.05,.45) node[gray_vertex] {} -- ++(-.3,.25) node[gray_vertex] {} -- ++(-.25,.3) node[gray_vertex] {} -- (v0);
			\node[anchor=north] at (0,-3) {(a) a generic tricycle graph;};
		\end{scope}
		\begin{scope}[xshift=5cm]
			\begin{scope}[yshift=-5mm,scale=1.1]
				\node[vertex,larger,label={[label distance=2pt]below:{\small$v_0$}}] (v0) at (0,0) {};
				\node[vertex,transition] (v1) at (0:1cm) {};
				\node[vertex,transition] (v2) at (60:1cm) {};
				\node[vertex,transition] (v3) at (120:1cm) {};
				\node[vertex,transition] (v4) at (180:1cm) {};
				\node[vertex,transition] (v5) at (240:1cm) {};
				\node[vertex,transition] (v6) at (300:1cm) {};
				\node[transition] at (0:1.325cm) {\small$\vv{1}{-}$};
				\node[transition] at (60:1.325cm) {\small$\vv{1}{+}$};
				\node[transition] at (120:1.325cm) {\small$\vv{2}{-}$};
				\node[transition] at (180:1.325cm) {\small$\vv{2}{+}$};
				\node[anchor=base,yshift=-2.2pt,transition] at (240:1.325cm) {\small$\vv{3}{-}$};
				\node[anchor=base,yshift=-2.2pt,transition] at (300:1.325cm) {\small$\vv{3}{+}$};
				\foreach \x in {1,...,6} {
					\draw (v0) -- (v\x);
				}
				\draw[transition] (v2) -- (v3) (v4) -- (v5) (v6) -- (v1);
				\draw (v1) to[bend left] (v2);
				\draw (v1) to[bend right] (v2);
				\draw (v3) to[bend left] (v4);
				\draw (v3) to[bend right] (v4);
				\draw (v5) to[bend left] (v6);
				\draw (v5) to[bend right] (v6);
			\end{scope}
			\node[anchor=north,align=left] at (0,-3) {(b) the minimal\\tricycle $\Tm$;};
		\end{scope}
		\begin{scope}[xshift=9cm]
			\begin{scope}[yshift=-2mm,scale=1.1]
				\node[vertex,larger,label={[label distance=2pt]below:{\small$v_0$}}] (v0) at (0,0) {};
				\node[vertex,transition] (v1) at (0:1cm) {};
				\node[vertex,transition] (v2) at (60:1cm) {};
				\node[vertex,transition] (v3) at (120:1cm) {};
				\node[vertex,transition] (v4) at (180:1cm) {};
				\node[vertex,transition] (v5) at (240:1cm) {};
				\node[vertex,transition] (v6) at (300:1cm) {};
				\draw (v1) ++(-30:.325) node[transition] {\small$\vv{1}{-}$};
				\draw (v2) ++(90:.325) node[transition] {\small$\vv{1}{+}$};
				\draw (v3) ++(90:.325) node[transition] {\small$\vv{2}{-}$};
				\draw (v4) ++(210:.325) node[transition] {\small$\vv{2}{+}$};
				\draw (v5) ++(210:.325) node[anchor=base,yshift=-2.2pt,transition] {\small$\vv{3}{-}$};
				\draw (v6) ++(-30:.325) node[anchor=base,yshift=-2.2pt,transition] {\small$\vv{3}{+}$};
				\draw (v2) ++(1,0) node[vertex] (v7) {};
				\draw (v3) ++(-1,0) node[vertex] (v8) {};
				\draw (v5) ++(-60:1cm) node[vertex] (v9) {};
				\foreach \x in {1,...,6} {
					\draw (v0) -- (v\x);
				}
				\draw (v1) -- (v2) (v3) -- (v4) (v5) -- (v6);
				\draw[transition] (v2) -- (v3) (v4) -- (v5) (v6) -- (v1);
				\draw (v1) -- (v7) -- (v2);
				\draw (v3) -- (v8) -- (v4);
				\draw (v5) -- (v9) -- (v6);
			\end{scope}
			\node[anchor=north,align=left] at (0,-3) {(c) the minimal\\ simple tricycle $\Tms$.};
		\end{scope}
	\end{tikzpicture}
	\caption{A generic tricycle, the minimal tricycle, and the minimal simple tricycle, with the transition vertices and the transition edges highlighted for emphasis.}
	\label{fig:tricycle}
\end{figure}

In what follows, we will show that every tricycle graph $G$ satisfies $\dgon(G) = 6$ and $\dgon(\Gamma(G)) = \dgon(\sigma_2(G)) = 5$.
First of all, we exhibit a positive rank divisor of degree $5$ on $\sigma_2(G)$.

\begin{proposition}
	\label{prop:special-divisor}
	Let $G$ be a tricycle graph. Then $\dgon(\sigma_2(G)) \leq 5$.
\end{proposition}
\begin{proof}
	Let $D_0 \in \Div(\sigma_2(G))$ be the effective divisor with two chips on $v_0$ and one chip on the midpoint of each of the transition edges $\vv{1}{+}\vv{2}{-}$, $\vv{2}{+}\vv{3}{-}$, $\vv{3}{+}\vv{1}{-}$. Then $\deg(D_0) = 5$. In light of \autoref{cor:discrete-rank-determining-set}, in order to show that $D_0$ has positive rank, it suffices to prove that $D_0$ reaches $v_0$ and the transition vertices $\vv{1}{-},\vv{1}{+},\vv{2}{-},\vv{2}{+},\vv{3}{-},\vv{3}{+}$.
	
	Clearly $D_0$ reaches $v_0$, for we have $D_0(v_0) > 0$. Now fix $i \in \{1,2,3\}$. To reach the transition vertices $\vv{i}{-}$ and $\vv{i}{+}$, let $S_i \subseteq V(\sigma_2(G))$ be the connected component of $\sigma_2(G) \setminus \supp(D_0)$ that contains the cycle $C_i$. Then the subset $\compl{S_i}$ can be fired, and doing so yields an effective divisor $D_i$ with $D_i(\vv{i}{-}) = D_i(\vv{i}{+}) = 1$.
	It follows that $D_0$ is a positive rank divisor on $\sigma_2(G)$, hence $\dgon(\sigma_2(G)) \leq 5$.
\end{proof}

Evidently, the divisor from \autoref{prop:special-divisor} is not supported on vertices of $G$. The remainder of this section is dedicated to showing that $G$ has no positive rank divisors of degree $5$. Along the way, we also prove that $\dgon(\Gamma(G)) \geq 5$.

In \autoref{lem:graad-5} below, we show that every positive rank $v_0$-reduced divisor of degree at most $5$ on a subdivision of the minimal tricycle $\Tm$ must be of a very specific form.
This will subsequently be used to show that $\dgon(\Gamma(G)) = \dgon(\sigma_2(G)) = 5$ (see \autoref{cor:dgon-metric}) and $\dgon(G) = 6$ (see \autoref{thm:2-subdivision-counterexample}) for every tricycle graph $G$.

For convenience, we use the following notation.

\begin{definition}
	\label{def:subdivision-path}
	Let $G$ be a graph and let $H$ be a subdivision of $G$.
	For $e = uw \in E(G)$, let $\pathcc{u}{w}^e \subseteq H$ denote the path $uv_1v_2\cdots v_kw$ in $H$ corresponding to the subdivided edge $e$.
	Furthermore, let $\pathoo{u}{w}^e := \pathcc{u}{w}^e \setminus \{u,w\}$, $\pathco{u}{w}^e := \pathcc{u}{w}^e \setminus \{w\}$ and $\pathoc{u}{w}^e := \pathcc{u}{w}^e \setminus \{u\}$ denote the corresponding open and half-open subpaths.
	If $e$ is the only edge between $u$ and $w$, then we omit the superscript and simply write $\pathcc{u}{w}$, $\pathoo{u}{w}$, $\pathco{u}{w}$ and $\pathoc{u}{w}$.
\end{definition}

The following simple lemma is essential to our proof, and will be used repeatedly.

\begin{lemma}
	\label{lem:helper-lemma}
	Let $G$ be a graph and let $v_0 \in V(G)$.
	Let $e_1,\ldots,e_k \in V(G)$ be the edges incident with $v_0$, and let $v_i \in V(G) \setminus \{v_0\}$ be the other endpoint of $e_i$ for every $i$.
	Moreover, let $H$ be a subdivision of $G$, let $D \in \Div(H)$ be a positive rank $v_0$-reduced divisor on $H$, and let $w \in V(H)$ be a vertex with $D(w) = 0$.
	Then an execution of Dhar's burning algorithm on the triple $(H,D,w)$ has the following properties:
	\begin{enumerate}[label=(\alph*)]
		\item\label{itm:helper:v0} $v_0$ is not burned;
		\item\label{itm:helper:neighbours} If $I \subseteq \{i\in [k] \, : \, \text{$v_i$ is burned}\}$, then $\bigcup_{i\in I} \pathco{v_0}{v_i}^{e_i}$ contains at least $|I|$ chips.
	\end{enumerate}
\end{lemma}
\begin{proof}
	\begin{enumerate}[label=(\alph*)]
		\item Since $D$ has positive rank and $D(w) = 0$, the divisor $D$ cannot be $w$-reduced, so Dhar's algorithm returns a non-empty subset $A \subseteq V(G)$ that can be fired. Since $D$ is $v_0$-reduced, we must have $v_0 \in A$, which means that $v_0$ is not burned.
		
		\item Partition $I$ as $I = I_0 \cup I_1$, where $i \in I_0$ if all vertices of the path $\pathoc{v_0}{v_i}^{e_i}$ are burned, and $i \in I_1$ otherwise.
		Since $v_0$ is not burned, it has at most $D(v_0)$ burning neighbours, so $|I_0| \leq D(v_0)$.
		Moreover, if $i\in I_1$, then $v_i$ is burned, but not all vertices of the path $\pathoc{v_0}{v_i}^{e_i}$ are burned, so there must be at least one chip on $\pathoo{v_0}{v_i}^{e_i}$.
		The conclusion follows.
		\qedhere
	\end{enumerate}
\end{proof}

We will apply \autoref{lem:helper-lemma} to an arbitrary subdivision of the minimal tricycle $\Tm$.
For this we use the following terminology.
Using notation from \autoref{def:subdivision-path}, if $H$ is a subdivision of $\Tm$, then the three transition edges $\vv{1}{+}\vv{2}{-}$, $\vv{2}{+}\vv{3}{-}$, $\vv{3}{+}\vv{1}{-}$ of $\Tm$ correspond to the paths $\pathcc{\vv{1}{+}}{\vv{2}{-}}$, $\pathcc{\vv{2}{+}}{\vv{3}{-}}$, $\pathcc{\vv{3}{+}}{\vv{1}{-}}$ in $H$, which we call the \emph{transition paths}.
The \emph{transition vertices} of $H$ are the images in $H$ of the original six transition vertices $\vv{1}{-},\vv{1}{+},\vv{2}{-},\vv{2}{+},\vv{3}{-},\vv{3}{+}$ of $\Tm$, or in other words, the endpoints of the transition paths in $H$.
(This is consistent with our definition of the transition vertices of a tricycle graph, which can also be seen as a subdivision of $\Tm$.)

\begin{lemma}
	\label{lem:graad-5}
	Let $H$ be a subdivision of the minimal tricycle $\Tm$. If $D \in \Div(H)$ is a positive rank $v_0$-reduced divisor with $\deg(D) \leq 5$, then $D$ must have two chips on $v_0$ and exactly one chip on each of the transition paths $\pathcc{\vv{1}{+}}{\vv{2}{-}}$, $\pathcc{\vv{2}{+}}{\vv{3}{-}}$, $\pathcc{\vv{3}{+}}{\vv{1}{-}}$.
\end{lemma}
\begin{proof}
	First, we prove that there must be at least one chip on every transition path. Suppose, for the sake of contradiction, that one of the transition paths, say $\pathcc{\vv{1}{+}}{\vv{2}{-}}$, has no chips at all. We start an execution of Dhar's burning algorithm on $(H,D,\vv{1}{+})$.
	Let $H_2^+ \subseteq H$ be the union of the cycle $C_2$ and the transition path $\pathcc{\vv{2}{+}}{\vv{3}{-}}$.
	We claim that the number of chips on $H_2^+$ plus the number of burned transition vertices in $H_2^+$ is at least $3$.
	To that end, note first of all that $\vv{2}{-}$ is burned, since there is no chip on the transition path $\pathcc{\vv{1}{+}}{\vv{2}{-}}$.
	Now we distinguish three cases:
	\begin{itemize}
		\item If $\vv{2}{+}$ is not burned, then there must be at least two chips on $C_2$ to stop the fire spreading from $\vv{2}{-}$ to $\vv{2}{+}$. In this case, $H_2^+$ contains at least one burned transition vertex (namely $\vv{2}{-}$) and at least two chips, for a total of at least $3$.
		\item If $\vv{2}{+}$ is burned but $\vv{3}{-}$ is not burned, then there must be at least one chip on the half-open transition path $\pathoc{\vv{2}{+}}{\vv{3}{-}}$. In this case, $H_2^+$ contains two burned transition vertices ($\vv{2}{-}$ and $\vv{2}{+}$) and at least one chip, for a total of at least $3$;
		\item If both $\vv{2}{+}$ and $\vv{3}{-}$ are burned, then $H_2^+$ contains three burned transition vertices ($\vv{2}{-}$, $\vv{2}{+}$ and $\vv{3}{-}$).
	\end{itemize}
	Likewise, write $H_1^- := C_1 \cup \pathcc{\vv{1}{-}}{\vv{3}{+}}$. Analogously, the number of chips plus the number of burned transition vertices on $H_1^-$ is at least $3$.
	Since $H_2^+$ and $H_1^-$ are disjoint, the total number of chips on the outer ring plus the total number of burned transition vertices is at least $6$.
	But since the transition vertices are exactly the $\Tm$-neighbours of $v_0$, and since the half-open paths $\pathco{v_0}{\vv{i}{\pm}}$ are disjoint from the outer ring, it follows from \myautoref{lem:helper-lemma}{itm:helper:neighbours} that $\deg(D) \geq 6$, which is a contradiction. We conclude that every transition path must have at least one chip.
	
	Second, we prove that there must be two chips on $v_0$.
	Since the total number of chips is at most $5$, there must be a cycle $C_i$ on the outer ring with at most one chip. Choose $w \in V(C_i)$ with $D(w) = 0$ and start an execution of Dhar's burning algorithm on $(H,D,w)$. Since there is at most one chip on $C_i$, the entire cycle $C_i$ is burned. It follows from \myautoref{lem:helper-lemma}{itm:helper:neighbours} that there are at least two chips on $\pathco{v_0}{\vv{i}{-}} \cup \pathco{v_0}{\vv{i}{+}}$. Therefore the number of chips on the outer ring is at most $3$, so there must be another cycle $C_j$ ($j \neq i$) on the outer ring with at most one chip. By an analogous argument, there are at least two chips on $\pathco{v_0}{\vv{j}{+}} \cup \pathco{v_0}{\vv{j}{-}}$. But since the outer ring has at least $3$ chips (one on every transition path), there can be at most $2$ chips on $\pathco{v_0}{\vv{i}{-}} \cup \pathco{v_0}{\vv{i}{+}} \cup \pathco{v_0}{\vv{j}{-}} \cup \pathco{v_0}{\vv{j}{+}}$. The only way to meet these requirements is if there are exactly two chips on $v_0$.
	
	To conclude the proof, note that $2$ chips on $v_0$ and at least $1$ chip on every transition path add up to at least $5$ chips in total. Since $\deg(D) \leq 5$, all chips have been accounted for. In particular, there cannot be more than one chip on each of the transition paths.
\end{proof}

\autoref{lem:graad-5} shows that every positive rank $v_0$-reduced divisor $D$ with $\deg(D) \leq 5$ must in fact satisfy $\deg(D) = 5$, so the following corollary is immediate.

\begin{corollary}
	\label{cor:dgon-subdivision}
	Let $H$ be a subdivision of the minimal tricycle $\Tm$. Then $\dgon(H) \geq 5$.
\end{corollary}

\begin{corollary}
	\label{cor:dgon-metric}
	Let $G$ be a tricycle graph. Then $\dgon(\Gamma(G)) = \dgon(\sigma_2(G)) = 5$.
\end{corollary}

\begin{proof}
	It follows from \autoref{prop:special-divisor} that $\dgon(\Gamma(G)) \leq \dgon(\sigma_2(G)) \leq 5$. Furthermore, since every subdivision of $G$ is a also subdivision of the minimal tricycle $\Tm$, it follows from \autoref{cor:dgon-subdivision} and \autoref{thm:two-conjectures} that
	\[ \dgon(\Gamma(G)) = \min_{k\in\N_1} \dgon(\sigma_k(G)) \geq 5.\qedhere \]
\end{proof}

All that remains is to prove that every tricycle graph has divisorial gonality $6$. To do so, we once again use the preceding lemmas.

\begin{theorem}
	\label{thm:2-subdivision-counterexample}
	Every tricycle graph $G$ satisfies $\dgon(G) = 6$.
\end{theorem}
\begin{proof}
	Suppose, for the sake of contradiction, that $\dgon(G) \leq 5$. Then we may choose a positive rank $v_0$-reduced divisor $D \in \Div(G)$ with $\deg(G) \leq 5$.
	We interpret $G$ as a subdivision of the minimal tricycle $\Tm$.
	It follows from \autoref{lem:graad-5} that $D$ has two chips on $v_0$ and exactly one chip on every transition path.
	Since $G$ is a tricycle graph, the transition edges of $\Tm$ are not subdivided. Therefore a chip on a transition edge must lie on one of the transition vertices.
		
	By the above, the divisor $D$ has between $0$ and $2$ chips on each of the cycles $C_1$, $C_2$, $C_3$ on the outer ring, and all such chips must lie on the transition vertices. Since the total number of chips on the outer ring is odd, there must be a cycle $C_i$ with exactly one chip. Assume without loss of generality that $C_1$ is such a cycle, and that $D(\vv{1}{-}) = 0$ and $D(\vv{1}{+}) = 1$.
	
	We start an execution of Dhar's burning algorithm on $(G,D,\vv{1}{-})$. Since there is only one chip on $C_1$, the entire cycle $C_1$ is burned. In particular, the vertex $\vv{1}{+}$ is burned. The transition edge $\vv{1}{+}\vv{2}{-}$ has exactly one chip, which is on $\vv{1}{+}$, so the fire spreads via this edge to the vertex $\vv{2}{-}$, which is also burned.
	But now we see that at least three $\Tm$-neighbours of $v_0$ are burned (namely, $\vv{1}{-}$, $\vv{1}{+}$ and $\vv{2}{-}$), so it follows from \myautoref{lem:helper-lemma}{itm:helper:neighbours} that there must be at least $3$ chips on $\pathco{v_0}{\vv{1}{-}} \cup \pathco{v_0}{\vv{1}{+}} \cup \pathco{v_0}{\vv{2}{-}}$. This is a contradiction, and we conclude that $\dgon(G) \geq 6$.
	
	To see that $\dgon(G) \leq 6$, note that the set $\{\vv{1}{-},\vv{1}{+},\vv{2}{-},\vv{2}{+},\vv{3}{-},\vv{3}{+}\}$ of all transition vertices is a strong separator. Therefore the effective divisor with one chip on each of the transition vertices has positive rank, by \autoref{thm:strong-separator}, so $\dgon(G) \leq 6$.
\end{proof}

This concludes the proof of validity of our counterexample.
In summary, every tricycle graph $G$ satisfies $\dgon(G) = 6$ and $\dgon(\Gamma(G)) = \dgon(\sigma_2(G)) = 5$.

\section{A family of examples with larger gaps}
\label{sec:skewered}
In this section, we combine tricycle graphs in a certain way in order to obtain graphs $G_k$ with $\dgon(G_k) = 6k$ and $\dgon(\sigma_2(G_k)) = \dgon(\Gamma(G_k)) = 5k$, which shows that the gap between $\dgon(\Gamma(G))$ and $\dgon(G)$ can be arbitrarily large.
Furthermore, we show that $\dgonr{\Gamma(G)}{r}$ and $\dgonr{G}{r}$ differ by at most a factor $2$.

\begin{definition}
	\label{def:skewered}
	Given a (connected) simple graph $H$ and an integer $t \geq 1$, an \emph{$(H,t)$-skewered graph} is a graph $G$ that can be obtained in the following way:
	\begin{itemize}
		\item Start with a disjoint union of graphs $G_1,\ldots,G_n$, where $n = |V(H)|$.
		\item For every $i\in [n]$, choose a base vertex $w_i \in V(G_i)$;
		\item For every edge $ij \in E(H)$, add $t$ parallel edges between $w_i$ and $w_j$, and subdivide these edges in an arbitrary way.
	\end{itemize}
	An example of a $(K_2,12)$-skewered graph is given in \autoref{fig:bipartite-skewered-tricycle} below.
\end{definition}

\begin{lemma}
	\label{lem:skewer}
	Let $G$ be an $(H,t)$-skewered graph with $t \geq \sum_{i=1}^{|V(H)|} \dgon(G_i)$. Then $\dgon(G) = \sum_{i=1}^{|V(H)|} \dgon(G_i)$.
\end{lemma}
\begin{proof}
	First, we prove that $\dgon(G) \leq \sum_{i=1}^{|V(H)|} \dgon(G_i)$. For every $i$, choose a positive rank divisor $D_i \in \Div(G_i)$ of minimum degree. This defines a divisor $D \in \Div(G)$ with $\deg(D) = \sum_{i=1}^{|V(H)|} \dgon(G_i)$. We prove that $D$ has positive rank. By \autoref{cor:discrete-rank-determining-set}, it suffices to prove that $D$ reaches all vertices of every $G_i$. Let $v \in V(G_i)$, and choose an effective divisor $D_i' \in \Div(G_i)$ equivalent to $D_i$ with $D_i'(v) > 0$. By \autoref{prop:increasing-firing-sets}, we can go from $D_i$ to $D_i'$ by subsequently firing an increasing sequence $U_1 \subseteq \cdots \subseteq U_k \subseteq V(G_i)$ of valid sets. Define $U_1' \subseteq \cdots \subseteq U_k' \subseteq V(G)$ by
	\[ U_j' := \begin{cases}
		U_j,&\text{if $w_i \notin U_j$};\\[1ex]
		U_j \cup \compl{V(G_i)},&\text{if $w_i \in U_j$}.
	\end{cases} \]
	Then, starting with $D$ and subsequently firing the sets $U_1' \subseteq \cdots \subseteq U_k'$, we obtain an equivalent divisor $D' = D - D_i + D_i' \in \Div(G)$. In other words, we can play the chip-firing game on $G_i$ while leaving the remainder of $G$ unchanged. This shows that $D$ reaches all vertices of every $G_i$, so it follows from \autoref{cor:discrete-rank-determining-set} that $\rank(D) \geq 1$.
	
	Next, we prove that $\dgon(G) \geq \sum_{i=1}^{|V(H)|} \dgon(G_i)$. Suppose, for the sake of contradiction, that $D \in \Div(G)$ is a positive rank $w_1$-reduced divisor with $\deg(D) < \sum_{i=1}^{|V(H)|} \dgon(G_i)$.
	We claim that $D$ is $w_i$-reduced for all $i$.
	To that end, let $S \subseteq V(G)$ be a subset for which there is some $ij \in E(H)$ with $w_i \in S$ and $w_j \notin S$. Since there are $t$ parallel paths in $G$ between $w_i$ and $w_j$, it follows from the max-flow min-cut theorem that $|E(S,\compl{S})| \geq t$. Therefore, $|E(S,\compl{S})| \geq t \geq \sum_{i=1}^{|V(H)|} \dgon(G_i) > \deg(D)$, so $S$ cannot be fired.
	Thus, if $S \subseteq V(G)$ is a subset which can be fired, then $w_1 \in S$ (because $D$ is $w_1$-reduced), and therefore $w_i \in S$ for all $i$ (because $H$ is connected).
	This proves our claim that $D$ is $w_i$-reduced for all $i$.
	
	Next, we claim that $D$ restricts to a positive rank divisor on every $G_i$. Indeed, let $v \in V(G_i)$ for some $i$, and choose an equivalent effective divisor $D' \in \Div(G)$ with $D'(v) > 0$. By \autoref{prop:increasing-firing-sets}, we can go from $D$ to $D'$ by subsequently firing an increasing sequence $U_1 \subseteq \cdots \subseteq U_k$ of valid sets. Since $D$ is $w_i$-reduced, we have $w_i \in U_1$, and therefore $w_i \in U_j$ for all $j$. Since $w_i$ is the only vertex in $G_i$ connected to anything outside of $G_i$, the firing sequence $U_1 \subseteq \cdots \subseteq U_k$ only ever sends chips out of $G_i$, and never into $G_i$. Hence it restricts to a valid firing sequence in $G_i$, which shows that the restricted divisor $D|_{G_i} \in \Div(G_i)$ reaches $v$. This proves our claim that $D$ restricts to a positive rank divisor on every $G_i$.
	But now it follows that $\deg(D) \geq \sum_{i=1}^{|V(H)|} \dgon(G_i)$, contrary to our assumption. This is a contradiction.
\end{proof}

\begin{proof}[{Proof of \autoref{thm:counterexample}}]
	Let $G_1,\ldots,G_k$ be tricycle graphs, and let $H$ be an arbitrary connected simple graph on $k$ vertices. Choose $t \geq 6k$, and let $G$ be an $(H,t)$-skewered graph obtained from the graphs $G_1,\ldots,G_k$. Then it follows from \autoref{lem:skewer} that $\dgon(G) = 6k$. Furthermore, for every $s \in \N_1$, the subdivided graph $\sigma_s(G)$ is an $(H,t)$-skewered graph relative to the base graphs $\sigma_s(G_1),\ldots,\sigma_s(G_k)$, so it follows from \autoref{lem:skewer} and \autoref{cor:dgon-metric} that
	\[ \dgon(\sigma_s(G)) = \sum_{i=1}^k \dgon(\sigma_s(G_i)) \geq \sum_{i=1}^k \dgon(\Gamma(G_i)) = 5k, \]
	with equality if $s = 2$. Therefore $\dgon(\Gamma(G)) = \dgon(\sigma_2(G)) = 5k$.
	
	A simple and bipartite realization can be obtained by choosing the tricycles $G_1,\ldots,G_k$ simple and bipartite (e.g. the tricycles skewered together in \autoref{fig:bipartite-skewered-tricycle}), and choosing an appropriate subdivision in the process of \autoref{def:skewered}.
\end{proof}

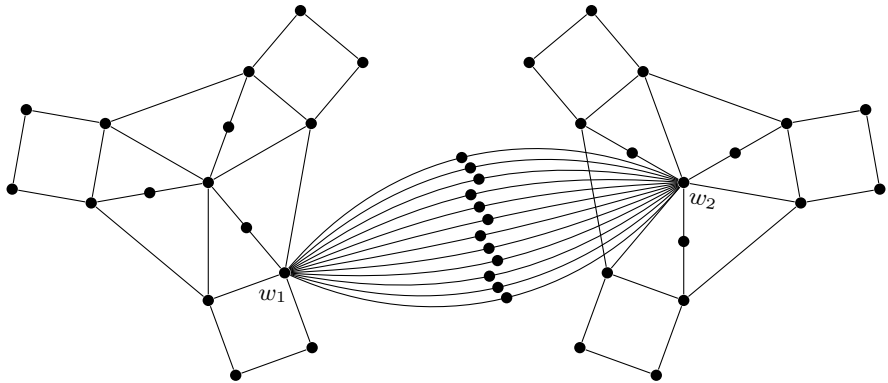
\begin{figure}[t]
	\centering
	\pgfmathsetmacro{\curschaal}{\schaal * 0.46}
	\begin{tikzpicture}[scale=\curschaal]
		\foreach \x in {1,2} {
			\pgfmathsetmacro{\xsh}{(\x - 1) * 8}
			\pgfmathsetmacro{\xrot}{(2 * \x - 3) * -20}
			\begin{scope}[xshift=\xsh cm,rotate=\xrot]
				\node[vertex] (v0\x) at (0,0) {};
				\node[vertex] (v1\x) at (10:2cm) {};
				\node[vertex] (v2\x) at (50:2cm) {};
				\node[vertex] (v3\x) at (130:2cm) {};
				\node[vertex] (v4\x) at (170:2cm) {};
				\node[vertex] (v5\x) at (250:2cm) {};
				\node[vertex] (v6\x) at (290:2cm) {};
				\draw (v1\x) ++(30:13.5mm) node[vertex] (v7\x) {};
				\draw (v2\x) ++(30:13.5mm) node[vertex] (v8\x) {};
				\draw (v3\x) ++(150:13.5mm) node[vertex] (v9\x) {};
				\draw (v4\x) ++(150:13.5mm) node[vertex] (v10\x) {};
				\draw (v5\x) ++(270:13.5mm) node[vertex] (v11\x) {};
				\draw (v6\x) ++(270:13.5mm) node[vertex] (v12\x) {};
				\draw (v1\x) -- (v2\x) -- (v3\x) -- (v4\x) -- (v5\x) -- (v6\x) -- (v1\x);
				\draw (v1\x) -- (v7\x) -- (v8\x) -- (v2\x);
				\draw (v3\x) -- (v9\x) -- (v10\x) -- (v4\x);
				\draw (v5\x) -- (v11\x) -- (v12\x) -- (v6\x);
				\draw (v0\x) -- (v1\x);
				\draw (v0\x) to node[vertex] {} (v2\x);
				\draw (v0\x) -- (v3\x);
				\draw (v0\x) to node[vertex] {} (v4\x);
				\draw (v0\x) -- (v5\x);
				\draw (v0\x) to node[vertex] {} (v6\x);
			\end{scope}
		}
		\draw (v61) ++(243:.44) node {\small$w_1$};
		\draw (v02) ++(-44:.45) node[rotate=-8] {\small$w_2$};
		\foreach \x in {1,...,12} {
			\pgfmathsetmacro{\hoek}{(\x - 6.5) * 6.7}
			\pgfmathsetmacro{\pos}{(mod(\x + 2,3) - 1) * -0.015 + 0.5}
			\draw (v02) to[bend left=\hoek] node[vertex,pos=\pos] {} (v61);
		}
	\end{tikzpicture}
	\caption{A simple, bipartite, $(K_2,12)$-skewered tricycle graph $G$ satisfying $\dgon(G) = 12$ and $\dgon(\Gamma(G)) = \dgon(\sigma_2(G)) = 10$.}
	\label{fig:bipartite-skewered-tricycle}
\end{figure}

\autoref{thm:counterexample} shows that the discrete and metric divisorial gonality can be arbitrarily far apart.
The following simple result shows that large gaps like this can only occur when the metric gonality is also large.
\begin{proposition}\label{prop:bound}
	Let $G$ be a graph. For every $r \geq 1$, one has $\dgonr{G}{r} \leq 2\dgonr{\Gamma(G)}{r} - r$.
\end{proposition}
\begin{proof}
	Let $D_1 \in \Div(\Gamma(G))$ be a divisor of rank $r$ and degree $d := \dgonr{\Gamma(G)}{r}$.
	Choose some $E \in \Div_+^r(G)$, and choose a divisor $D_1' \sim D_1$ such that $D_1' \geq E$.
	Let $D_2 \in \Div(G)$ be the divisor obtained from $D_1'$ by replacing every chip on the interior of some edge $uv \in E(G)$ by one chip on $u$ and one chip on $v$.
	Since $D_1' \geq E$ and $\supp(E) \subseteq V(G)$, the divisor $D_1'$ has at least $r$ chips on vertices of $G$, so $\deg(D_2) \leq 2d - r$.
	By firing everything but the interior of the edge $uv$, we can move the newly added chips on $u$ and $v$ so that one of the two reaches the original position of the chip in $D_1'$ and the other becomes superfluous.
	This shows that $D_2$ is equivalent on $\Gamma$ to a divisor $D_2'$ with $D_2' \geq D_1'$, so $\rank_G(D_2) = \rank_\Gamma(D_2) \geq r$, by \cite[Thm.~1.3]{Hladky-Kral-Norine}.
\end{proof}

\section{Computational results and open questions}
\label{sec:closing-remarks}

Apart from the tricycle graphs, we have found a few other counterexamples, which we sketch here.
First of all, the proofs from \mysecref{sec:main-counterexample} still hold if each of the cycles $C_1$, $C_2$ and $C_3$ is replaced by any graph $C$ which has two distinct vertices $\vv{}{-},\vv{}{+}$ such that: (i) there are two edge-disjoint paths between $\vv{}{-}$ and $\vv{}{+}$; (ii) the divisor $\vv{}{-} + \vv{}{+}$ has positive rank on $C$.

Second, we have found a number of counterexamples which we have verified computationally, but for which we have no formal proof.
Most of these have a structure very similar to a tricycle graph: there are $3$ cycles which are connected to one another and to a central vertex in some way.
A small selection of these counterexamples is given in \autoref{fig:additional-counterexamples}.
In each of these, the optimal divisor on the $2$-regular subdivision $\sigma_2(G)$ has $3$ chips on the midpoints of certain edges, and $2$ or $3$ chips on the central vertex.
Note that the counterexample depicted in \myautoref*{fig:additional-counterexamples}{c} is $3$-regular.
We have also found counterexamples where the outer ring has $5$ or $7$ cycles; see \myautoref*{fig:additional-counterexamples}{d}.
We have not found a counterexample with $9$ or more cycles on the outer ring.
See \cite{gonality-code} for code and additional figures.

\begin{figure}[h!t]
	\centering
	\pgfmathsetmacro{\curschaal}{\schaal * 0.5}
	\def\chipkleur{cyan!90!black}
	\def\mylabelx{-2}
	\def\mylabely{-1.5}
	\def\abshift{2mm}
	\def\efghschaal{1.25}
	\def\efghshift{.5mm}
	\begin{tikzpicture}[scale=\curschaal,
	                    yscale=-1,
	                    vertex/.style={circle,inner sep=1.3pt,fill=black!70},
	                    edge_chip/.style={regular polygon,regular polygon sides=6,fill=\chipkleur,inner sep=1.65pt},
	                    vertex_chip/.style={edge_chip}]
		\begin{scope}
			\begin{scope}[yshift=\abshift,rotate=-120]
				\node[vertex] (v0) at (0,0) {};
				\node[vertex] (v1) at (30:1cm) {};
				\node[vertex] (v2) at (300:1cm) {};
				\node[vertex] (v3) at (240:1cm) {};
				\node[vertex] (v4) at (180:1cm) {};
				\node[vertex] (v5) at (120:1cm) {};
				\draw (v3) ++(-60:1cm) node[vertex] (v6) {};
				\draw (v5) ++(-1,0) node[vertex] (v7) {};
				\begin{scope}[shift=(v1)]
					\node[vertex] (v8) at (0:1cm) {};
					\node[vertex] (v9) at (60:1cm) {};
				\end{scope}
				\draw (v0) to[bend left=15] (v1);
				\draw (v0) to[bend right=15] (v1);
				\draw (v0) to[bend left=15] (v2);
				\draw (v0) to[bend right=15] (v2);
				\draw (v0) to[bend left=15] (v5);
				\draw (v0) to[bend right=15] (v5);
				\draw (v1) to node[edge_chip] {} (v2);
				\draw (v5) to node[edge_chip] {} (v1);
				\draw (v3) to node[edge_chip] {} (v4);
				\draw (v9) -- (v8) -- (v1) -- (v9);
				\draw (v6) -- (v2) -- (v3) -- (v6);
				\draw (v7) -- (v4) -- (v5) -- (v7);
			\end{scope}
			\draw (v0) ++(0,.105) node[vertex_chip] {};
			\draw (v0) ++(0,-.105) node[vertex_chip] {};
			\node[anchor=base west] at (\mylabelx,\mylabely) {(a)};
		\end{scope}
		
		\begin{scope}[xshift=4cm]
			\begin{scope}[yshift=\abshift,rotate=-120]
				\node[vertex] (v9) at (0,-.1) {};
				\node[vertex] (v1) at (240:1cm) {};
				\node[vertex] (v2) at (300:1cm) {};
				\begin{scope}[shift=(v1)]
					\node[vertex] (v0) at (-60:1cm) {};
				\end{scope}
				\node[vertex] (v3) at (30:1cm) {};
				\begin{scope}[shift=(v3)]
					\node[vertex] (v5) at (60:1cm) {};
					\node[vertex] (v6) at (0:1cm) {};
				\end{scope}
				\node[vertex] (v4) at (150:1cm) {};
				\begin{scope}[shift=(v4)]
					\node[vertex] (v7) at (120:1cm) {};
					\node[vertex] (v8) at (180:1cm) {};
				\end{scope}
				\draw (v0) -- (v1) -- (v2) -- (v0);
				\draw (v1) to node[edge_chip] {} (v4) -- (v7) -- (v8) -- (v4) to node[edge_chip] {} (v3) to node[edge_chip] {} (v2);
				\draw (v3) -- (v5) -- (v6) -- (v3);
				\draw (v2) to[bend left=15] (v9);
				\draw (v2) to[bend right=15] (v9);
				\draw (v3) to[bend left=15] (v9);
				\draw (v3) to[bend right=15] (v9);
				\draw (v4) to[bend left=15] (v9);
				\draw (v4) to[bend right=15] (v9);
			\end{scope}
			\draw (v9) ++(0,.105) node[vertex_chip] {};
			\draw (v9) ++(0,-.105) node[vertex_chip] {};
			\node[anchor=base west] at (\mylabelx,\mylabely) {(b)};
		\end{scope}
		
		\begin{scope}[xshift=8cm]
			\begin{scope}[scale=.86,yshift=-1.8mm]
				\node[vertex] (v0) at (0,0) {};
				\draw (v0) ++(30:6mm) node[vertex] (v1) {};
				\draw (v1) ++(-20:8mm) node[vertex] (v2) {};
				\draw (v1) ++(80:8mm) node[vertex] (v3) {};
				\draw (v2) ++(10:6mm) node[vertex] (v4) {};
				\draw (v3) ++(50:6mm) node[vertex] (v5) {};
				\draw (v0) ++(150:6mm) node[vertex] (v6) {};
				\draw (v6) ++(100:8mm) node[vertex] (v7) {};
				\draw (v6) ++(200:8mm) node[vertex] (v8) {};
				\draw (v7) ++(130:6mm) node[vertex] (v9) {};
				\draw (v8) ++(170:6mm) node[vertex] (v10) {};
				\draw (v0) ++(270:6mm) node[vertex] (v11) {};
				\draw (v11) ++(220:8mm) node[vertex] (v12) {};
				\draw (v11) ++(320:8mm) node[vertex] (v13) {};
				\draw (v12) ++(250:6mm) node[vertex] (v14) {};
				\draw (v13) ++(290:6mm) node[vertex] (v15) {};
				\draw (v0) -- (v1) -- (v2) -- (v3) -- (v1);
				\draw (v2) -- (v4) -- (v5) -- (v3);
				\draw (v0) -- (v6) -- (v7) -- (v8) -- (v6);
				\draw (v7) -- (v9) -- (v10) -- (v8);
				\draw (v0) -- (v11) -- (v12) -- (v13) -- (v11);
				\draw (v12) -- (v14) -- (v15) -- (v13);
				\draw (v5) to node[edge_chip] {} (v9);
				\draw (v10) to node[edge_chip] {} (v14);
				\draw (v15) to node[edge_chip] {} (v4);
			\end{scope}
			\draw (v0) ++(.105,0) node[vertex_chip] {};
			\draw (v0) ++(-.105,0) node[vertex_chip] {};
			\node[anchor=base west] at (\mylabelx,\mylabely) {(c)};
		\end{scope}
		
		\begin{scope}[xshift=12cm]
			\begin{scope}[scale=.75,yshift=-2.1mm,yscale=-1]
				\def\numcykels{5}
				\pgfmathtruncatemacro{\outerring}{2*\numcykels}
				\node[vertex] (v0) at (0,0) {};
				\foreach \x in {1,...,\outerring} {
					\pgfmathsetmacro{\hoek}{(\x - 1) / \outerring * 360 + 3.5 * (-1)^(\x - 1)}
					\node[vertex] (v\x) at (\hoek:2cm) {};
				}
				\foreach \x in {1,...,\outerring} {
					\draw (v0) -- (v\x);
					\pgfmathtruncatemacro{\y}{mod(\x,\outerring) + 1}
					\ifodd\x
						\draw (v\x) to[bend left=20] (v\y);
						\draw (v\x) to[bend right=20] (v\y);
					\else
						\draw (v\x) to[bend right=5] node[edge_chip] {} (v\y);
					\fi
				}
			\end{scope}
			\draw (v0) ++(0,0.105) node[vertex_chip] {};
			\draw (v0) ++(0,-0.105) node[vertex_chip] {};
			\node[anchor=base west] at (\mylabelx,\mylabely) {(d)};
		\end{scope}
		
		\begin{scope}[yshift=4cm]
			\begin{scope}[scale=\efghschaal,yshift=\efghshift]
				\node[vertex] (v9) at (0,0) {};
				\node[vertex] (v3) at (10:8mm) {};
				\node[vertex] (v6) at (50:8mm) {};
				\node[vertex] (v2) at (130:8mm) {};
				\node[vertex] (v5) at (170:8mm) {};
				\node[vertex] (v1) at (250:8mm) {};
				\node[vertex] (v7) at (290:8mm) {};
				\begin{scope}[shift=(v3)]
					\node[vertex] (v0) at (50:8mm) {};
				\end{scope}
				\begin{scope}[shift=(v2)]
					\node[vertex] (v8) at (170:8mm) {};
				\end{scope}
				\begin{scope}[shift=(v1)]
					\node[vertex] (v4) at (290:8mm) {};
				\end{scope}
				\draw (v8) to[bend left=15] node[edge_chip] {} (v0);
				\draw (v7) to[bend right=20] node[edge_chip] {} (v0);
				\draw (v8) to[bend right=20] node[edge_chip] {} (v1);
				\draw (v2) -- (v8) (v0) -- (v3) -- (v9) -- (v1) (v8) -- (v5) -- (v9) -- (v2) -- (v5);
				\draw (v3) -- (v6) -- (v0) (v7) -- (v1) -- (v4) -- (v7) -- (v9) -- (v6);
			\end{scope}
			\draw (v9) ++(0,0.105) node[vertex_chip] {};
			\draw (v9) ++(0,-0.105) node[vertex_chip] {};
			\node[anchor=base west] at (\mylabelx,\mylabely) {(e)};
		\end{scope}
		
		\begin{scope}[xshift=4cm,yshift=4cm]
			\begin{scope}[scale=\efghschaal,yshift=\efghshift]
				\node[vertex] (v9) at (0,0) {};
				\node[vertex] (v3) at (10:8mm) {};
				\node[vertex] (v0) at (50:8mm) {};
				\node[vertex] (v2) at (130:8mm) {};
				\node[vertex] (v5) at (170:8mm) {};
				\node[vertex] (v1) at (250:8mm) {};
				\node[vertex] (v4) at (290:8mm) {};
				\begin{scope}[shift=(v3)]
					\node[vertex] (v6) at (50:8mm) {};
				\end{scope}
				\begin{scope}[shift=(v2)]
					\node[vertex] (v8) at (170:8mm) {};
				\end{scope}
				\begin{scope}[shift=(v1)]
					\node[vertex] (v7) at (290:8mm) {};
				\end{scope}
				\draw (v8) to[bend left=20] node[edge_chip] {} (v0);
				\draw (v7) to[bend right=30] node[edge_chip] {} (v0);
				\draw (v8) to[bend right=20] node[edge_chip] {} (v1);
				\draw (v0) -- (v3) -- (v9) -- (v0) -- (v6) -- (v3);
				\draw (v2) -- (v8) -- (v5) -- (v9) -- (v1) -- (v4) -- (v9) -- (v2) -- (v5);
				\draw (v4) -- (v7) -- (v1);
			\end{scope}
			\draw (v9) ++(0,0.105) node[vertex_chip] {};
			\draw (v9) ++(0,-0.105) node[vertex_chip] {};
			\node[anchor=base west] at (\mylabelx,\mylabely) {(f)};
		\end{scope}
		
		\begin{scope}[xshift=8cm,yshift=4cm]
			\begin{scope}[scale=\efghschaal,yshift=\efghshift]
				\node[vertex] (v9) at (0,0) {};
				\node[vertex] (v7) at (10:8mm) {};
				\node[vertex] (v2) at (50:8mm) {};
				\node[vertex] (v8) at (130:8mm) {};
				\node[vertex] (v1) at (170:8mm) {};
				\node[vertex] (v6) at (250:8mm) {};
				\node[vertex] (v0) at (290:8mm) {};
				\begin{scope}[shift=(v7)]
					\node[vertex] (v5) at (50:8mm) {};
				\end{scope}
				\begin{scope}[shift=(v8)]
					\node[vertex] (v4) at (170:8mm) {};
				\end{scope}
				\begin{scope}[shift=(v6)]
					\node[vertex] (v3) at (290:8mm) {};
				\end{scope}
				\draw (v2) to[bend right=15] node[edge_chip] {} (v8);
				\draw (v1) to[bend right=15] node[edge_chip] {} (v6);
				\draw (v0) to[bend right=15] node[edge_chip] {} (v7);
				\draw (v3) -- (v0) -- (v9) -- (v1) -- (v8);
				\draw (v2) -- (v9) -- (v3) -- (v6) -- (v0);
				\draw (v7) -- (v9) -- (v4) -- (v8) -- (v9) -- (v5) -- (v2) -- (v7) -- (v5);
				\draw (v4) -- (v1);
				\draw (v6) -- (v9);
			\end{scope}
			\draw (v9) ++(-30:0.13cm) node[vertex_chip] {};
			\draw (v9) ++(-150:0.13cm) node[vertex_chip] {};
			\draw (v9) ++(90:0.13cm) node[vertex_chip] {};
			\node[anchor=base west] at (\mylabelx,\mylabely) {(g)};
		\end{scope}
		
		\begin{scope}[xshift=12cm,yshift=4cm]
			\begin{scope}[scale=\efghschaal,yshift=\efghshift]
				\node[vertex] (v9) at (0,0) {};
				\node[vertex] (v3) at (10:8mm) {};
				\node[vertex] (v6) at (50:8mm) {};
				\node[vertex] (v2) at (130:8mm) {};
				\node[vertex] (v5) at (170:8mm) {};
				\node[vertex] (v1) at (250:8mm) {};
				\node[vertex] (v7) at (290:8mm) {};
				\begin{scope}[shift=(v3)]
					\node[vertex] (v0) at (50:8mm) {};
				\end{scope}
				\begin{scope}[shift=(v2)]
					\node[vertex] (v8) at (170:8mm) {};
				\end{scope}
				\begin{scope}[shift=(v1)]
					\node[vertex] (v4) at (290:8mm) {};
				\end{scope}
				\draw (v8) to[bend left=15] node[edge_chip] {} (v0);
				\draw (v7) to[bend right=20] node[edge_chip] {} (v0);
				\draw (v8) to[bend right=20] node[edge_chip] {} (v1);
				\draw (v0) -- (v3) -- (v9) -- (v0) -- (v6) -- (v9) -- (v1);
				\draw (v7) -- (v1) -- (v4) -- (v9) -- (v2) -- (v8) -- (v5) -- (v9) -- (v7) -- (v4);
				\draw (v2) -- (v5);
				\draw (v3) -- (v6);
				\draw (v8) -- (v9);
			\end{scope}
			\draw (v9) ++(-30:0.13cm) node[vertex_chip] {};
			\draw (v9) ++(-150:0.13cm) node[vertex_chip] {};
			\draw (v9) ++(90:0.13cm) node[vertex_chip] {};
			\node[anchor=base west] at (\mylabelx,\mylabely) {(h)};
		\end{scope}
	\end{tikzpicture}
	\caption{Additional counterexamples to \protect\myautoref{conj:subdivision-metric}{itm:conj:subdivision} for $k = 2$ and $r = 1$. The small blue hexagons represent the chips of an optimal divisor on the $2$-regular subdivision.}
	\label{fig:additional-counterexamples}
\end{figure}
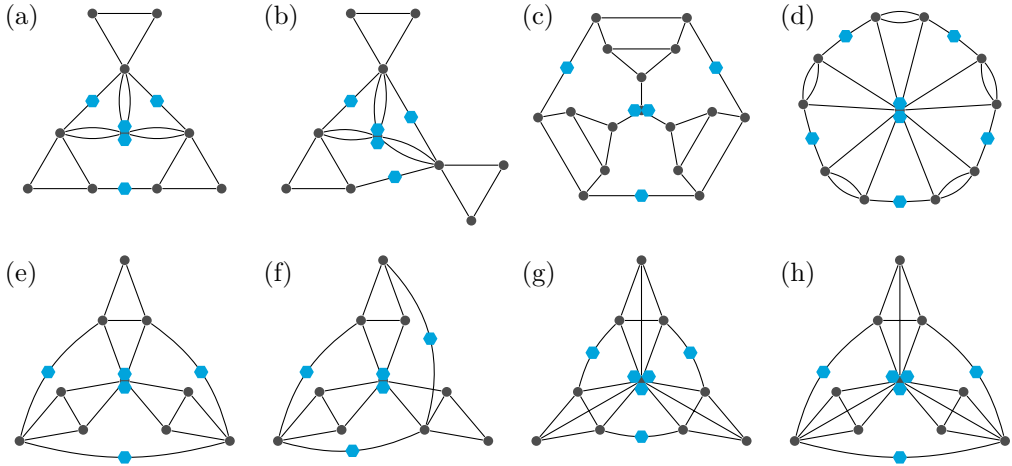

We have tested \myautoref{conj:subdivision-metric}{itm:conj:subdivision} for $k = 2$ and $r = 1$ for all simple connected graphs on at most $10$ vertices.
These graphs were generated using the program \texttt{geng} from the \texttt{gtools} suite packaged with \texttt{nauty} \cite{nauty-paper, nauty}, and tested using custom code that we wrote to compute the divisorial gonality of a graph \cite{gonality-code}.
We have found that every simple connected graph with $9$ or fewer vertices satisfies $\dgon(\sigma_2(G)) = \dgon(G)$, and that there are exactly $29$ counterexamples with $10$ vertices (and no parallel edges), including the minimal simple tricycle $\Tms$ and the graphs depicted in \hyperref[fig:additional-counterexamples]{\autoref*{fig:additional-counterexamples}(e)--(h)}.
For a list of all $29$ minimal simple counterexamples and code to reproduce this list, see \cite{gonality-code}.
There we have also included optimized code to check whether the divisorial gonality of a given graph satisfies the Brill--Noether bound, which we have used to verify \autoref{conj:Brill-Noether} for all simple connected graphs with at most $13$ vertices.
No counterexamples were found.

We close with a few open problems.

\begin{enumerate}[label=\arabic*.]
	\item As mentioned before, the Brill--Noether conjecture \cite[Conj.~3.9(1)]{Baker-specialization} remains open.
	
	\item What is the smallest constant $c$ such that $\dgon(G) \leq c \dgon(\Gamma(G))$ for all graphs $G$?
	Our examples from \autoref{thm:counterexample} show that $c \geq \frac{6}{5}$, and \autoref{prop:bound} shows that $c \leq 2$.
	
	\item All counterexamples presented in this paper satisfy $\dgon(\sigma_2(G)) < \dgon(G)$.
	Note that this implies that $\dgon(\sigma_k(G)) < \dgon(G)$ for every even number $k$.
	Is there a graph $G$ such that $\dgon(\sigma_k(G)) < \dgon(G)$ for some odd number $k$?
	Is there a graph $G$ such that $\dgon(\sigma_2(G)) = \dgon(G)$ but $\dgon(\sigma_k(G)) < \dgon(G)$ for some $k > 2$?
	
	\item Is there a graph $G$ such that $\dgon(\Gamma(G)) = \dgon(G)$, but $\dgonr{\Gamma(G)}{r} < \dgonr{G}{r}$ for some $r \geq 2$?
\end{enumerate}

\small
\paragraph{Acknowledgements}
We are grateful to Dion Gijswijt, Sophie Huiberts and Alejandro Vargas for helpful discussions.
The first author is partially supported by the Dutch Research Council (NWO), project number 613.009.127.
The second author would like to thank the Max Planck Institute for Mathematics Bonn for its financial support.

\bibliographystyle{alpha}
\bibliography{bronnen}

\end{document}